\documentclass[reqno]{amsart}
\usepackage{cite, amssymb}
\newtheorem{thm}{Theorem}
\newtheorem{cor}{Corollary}
\newtheorem{lem}{Lemma}
\newtheorem{prop}{Proposition}
\theoremstyle{definition}
\newtheorem{defn}{Definition}
\theoremstyle{remark}
\newtheorem{rem}{Remark}
\usepackage{color}

\newcommand{\supp}{\mathop{\rm supp}}
\newcommand{\Real}{\mathbb R}
\newcommand{\eps}{\varepsilon}
\renewcommand{\kappa}{\varkappa}
\newcommand{\lm}{\lambda}

\newcommand{\cH}{\mathcal{H}}
\newcommand{\cI}{\mathcal{I}}
\newcommand{\dmn}{\mathop{\rm dom}}
\newcommand{\en}{{\lm\nu}}

\begin{document}

\title[On coupling constant thresholds in one dimension]{On coupling constant thresholds\\ in one dimension}


\author{Yuriy Golovaty}%
\address{Department of Mechanics and Mathematics,
  Ivan Franko National University of Lviv\\
  1 Universytetska str., 79000 Lviv, Ukraine}
\curraddr{}
\email{yuriy.golovaty@lnu.edu.ua}


\begin{abstract}
The threshold behaviour of negative eigenvalues for Schr\"{o}dinger operators of the type
\begin{equation*}
H_\lm=-\frac{d^2}{dx^2}+U+ \lm\alpha_\lm V(\alpha_\lm \cdot)
\end{equation*}
is considered. The potentials $U$ and $V$ are real-valued bounded functions of compact support, $\lm$ is a  positive parameter, and positive sequence $\alpha_\lm$ has a finite or infinite limit as $\lm\to 0$. Under certain conditions on the potentials there exists
a bound state of $H_\lm$ which is absorbed at the bottom
of the continuous spectrum. For several cases of the limiting behaviour of  sequence~$\alpha_\lm$, asymptotic formulas for the  bound states are proved and the first order terms are computed explicitly.
\end{abstract}

\maketitle

\section{Introduction}

In the seventies and eighties  of the last century there has been a considerable interest in the study of the low-energy behaviour of Hamiltonians $\cH_\lm=-\Delta+\lm V$, especially of the small $\lm$ behavior of  bound states and zero-energy resonances, as well in the study of the absorption of eigenvalues by the continuous spectrum.
It is known that in three dimensions there is no bound state of $\cH_\lm$, if $V$ is a sufficiently shallow well. In contrast to this case, a suitable short-range potential  in one or two dimensions can produce a bound state for all small $\lm$. For the case of the one-dimensional Hamiltonian $-\frac{d^2}{dx^2}+\lm V$, Simon~\cite{Simon:1976} proved that if $\int_{\Real}V(x)\,dx\leq0$ and $\int_{\Real}(1+|x|^2)|V(x)|\,dx<\infty$, then the operator has a unique negative eigenvalue $e_\lm$ for all positive $\lm$ such that $e_\lm$ approaches zero  as $\lm\to 0$.
This result has been extended by Klaus \cite{Klaus:1977} to the class of potentials $V$ obeying $\int_{\Real}(1+|x|)|V(x)|\,dx<\infty$.
Moreover, if $\int_{\Real}e^{a|x|}|V(x)|\,dx<\infty$, for instance, $V$ is of compact support, then the eigenvalue $e_\lm$ is analytic in $\lm$ at $\lm=0$.

The question of how negative eigenvalues are absorbed at the bottom of the essential spectrum  has been discussed by many authors \cite{BlankenbeclerGoldbergerSimon1977, Rauch1980, KlausSimonAnPh:1980, KlausSimonComMathPh:1980, AlbeverioGesztesyHoegh-Krohn:1982, Holden1985, BolleGesztesyWilk:1985, GesztesyHolden:1987, JensenMelgaard2002, Gadylshin2002, Gadylshin2004, BorisovGadylshin2006}. It is worth noting that the threshold behaviour is also strongly dependent on the dimension of space~\cite{ KlausSimonAnPh:1980}. In two dimensions $e_\lm$ is exponentially small  as $\lm\to 0$; in one dimension, $e_\lm\sim c\lm^2$ with $c\neq0$ if $\int_{\Real}V\,dx<0$ and $e_\lm\sim c\lm^4$ if $\int_{\Real}V\,dx=0$.
The threshold behaviour as a general perturbation phenomenon has been investigated in \cite{Simon:1977, Klaus:1982}, where in particular one gives an answer  to the question what will happen with the eigenvalue $e_\lm$ in the limit -- will it disappear or remain as an eigenvalue embedded in the continuous spectrum?

In \cite{Klaus:1982},  some general
theorems about the existence and threshold behavior of eigenvalues for self-adjoint operators of the form $A+\lm B$ have been applied to several special cases. One of them has been concerned with the threshold behavior of Hamiltonian $-\frac{d^2}{dx^2}+U+\lm V$, where $U$ and $V$ are  of compact support. The Schrodinger operator has a small negative bound state in the limit of weak coupling (not necessarily a unique one), if a certain  relation between the potentials $U$ and $V$ holds. Namely,  $-\frac{d^2}{dx^2}+U+\lm V$ has the coupling constant threshold $\lm=0$,
if ope\-ra\-tor $-\frac{d^2}{dx^2}+U$ possesses a zero-energy resonance with half-bound state $h$ and $\int_\Real Vh^2\,dx<0$. A unique  eigenvalue $e_\lm$ that converges to zero as $\lm\to 0$ is analytic at $\lm=0$ and obeys
\begin{equation*}
  e_\lm=-\lm^2\left(\frac{1}{h_+^2+h_-^2}\int_{\Real}Vh^2\,dx\right)^2
  +O(\lm^3)\quad\text{as }\lm\to 0,
\end{equation*}
where $h_\pm=\lim\limits_{x\to\pm\infty}h(x)$.
If $\int_\Real Vh^2\,dx>0$,
$\lm=0$ is not a coupling constant threshold. If $\int_\Real Vh^2\,dx=0$ and $\supp V$ lies between two consecutive zeros of $h$, then there exists a bound state near zero for all small enough $\lm$  (positive and negative). The precise definitions of the coupling constant threshold, zero-energy resonances and half-bound states  will be given in the next section.
Klaus  seems to have been the first to notice the close connection between zero-energy resonances of the unperturbed operator  and the threshold
phenomenon for the operator subject to certain short-range perturbations.
For a treatment of the threshold phenomena for Hamiltonians with periodic potentials perturbed by short range ones  we refer the reader to
\cite{ Klaus:1982, GesztesySimon1993, FassariKlaus1998}.

For the spectral theory, the threshold behaviour of eigenvalues and resonances of Schrodinger operators as they are absorbed by the continuous spectrum is inte\-re\-sting in itself. However the study of this phenomenon is relevant to the stability of solutions for the Korteweg-de~Vries equation \cite{ScharfWreszinski:1981} as well as to the existence of 'breathers' -- nonlinear waves in which energy concentrates in a localized and oscillatory fashion --  for discrete nonlinear Schr\"{o}dinger systems \cite{Weinstein:1999, KevrekidisRasmussenBishop2001}.

The aim of this paper is a detailed analysis of the threshold behaviour of negative eigenvalues for 1D Schrodinger operators having the form
\begin{equation*}
H_\lm=-\frac{d^2}{dx^2}+U(x)+ \lm\alpha_\lm V(\alpha_\lm x),
\end{equation*}
$\dmn H_\lm=W_2^2(\Real)$.
Here $U$ and $V$ are $L^\infty(\Real)$-functions of compact support and $\lm$ is a positive parameter. We assume that positive sequence $\{\alpha_\lm\}_{\lm>0}$ has a finite or infinite limit as $\lm\to 0$.
The perturbation $\lm\alpha_\lm V(\alpha_\lm \cdot)$ of operator $H_0=-\frac{d^2}{dx^2}+U$ combines the scaling of potential $V$ in two directions. Regarding the product $\lm\alpha_\lm$ as a coupling constant, we see that this constant can both be infinitely small and infinitely large as $\lm\to 0$; the support of the perturbation can shrink to a point in the limit if $\alpha_\lm\to+\infty$,  extend to the whole line if $\alpha_\lm\to 0$, or  remain actually a fixed size if $\alpha_\lm$ tends to a positive number.
The spectrum  of $H_\lm$ consists of at most a finite number of negative eigenvalues and the essential spectrum $[0, \infty)$. Under certain conditions on the potentials $U$ and $V$  the operator $H_\lm$  has a unique bound state $e_\lm$ that is absorbed at the bottom of the essential spectrum as $\lm$ goes to zero (as $\lm$ is continued in the opposite direction the bound state emerges from the essential spectrum). The threshold bound state $e_\lm$ may or may not be the ground state.
Our method is different from that of Simon and Klaus; we don't use the Birman-Schwinger principle. We obtain a part of our results from the  norm resolvent convergence of Schr\"{o}dinger operators with $(a \delta'+b \delta)$-like potentials \cite{GolovatyHrynivJPA:2010, Golovaty:2012, GolovatyHrynivProcEdinburgh2013, GolovatyIEOT2013};  to prove the rest ones we use the asymptotic method of quasimodes  or in other words of ``almost" eigen\-va\-lues and eigenfunctions.

\section{Main Results}
In this section we give the main results of this paper. We start with some definitions.

\begin{defn}
   Let $A$ and $B_\lm$ be self-adjoint operators and $B_\lm$ be relatively $A$-com\-pact for all $\lm>0$ (then $\sigma_{ess}(A+B_\lm)=\sigma_{ess}(A)$).  Suppose that interval $(a,b)$ is a gap in the spectrum of $A$ and $b\in\sigma_{ess}(A)$. If  we can find an eigenvalue $e_\lm$ of $A+B_\lm$ in gap $(a,b)$ for all $\lm>0$ with the property that  $e_\lm\to b-0$  as  $\lm\to 0$, then we call $\lm=0$ the \textit{coupling constant threshold}.
\end{defn}

There are different definitions of the zero-energy resonances for Schr\"{o}dinger operators, but in  one dimension it is convenient for us to use the following one.

\begin{defn}
We say  operator~$H_0=-\frac{d^2}{dx^2}+U$  possesses  a \emph{zero-energy resonance} if there exists a non trivial solution~$h$ of the equation $ -h'' + Uh= 0$
that is bounded on the whole line. We then call $h$ the \emph{half-bound state}.
\end{defn}

Set $h_\pm=\lim\limits_{x\to\pm\infty}h(x)$;
the limits exist, because the half-bound state is constant outside $\supp U$ as a bounded solution of equation $h''=0$. Also, both the values $h_\pm$ are different from zero, and so $h\not\in L_2(\Real)$. Since a half-bound state is defined up to a scalar multiplier, we renormalize it by defining
$u=h/h_-$. We call $u$ the \textit{normalized half-bound state}.
Let $\theta$ hereafter denote the limit of the  normalized half-bound state as $x\to+\infty$, i.e., $\theta:=\lim_{x\to+\infty}u(x)$,
provided $\lim_{x\to-\infty}u(x)=1$.

\begin{thm}\label{TheoremOrder2}
Suppose that $H_0=-\frac{d^2}{dx^2}+U$ has a zero-energy resonance with normalized half-bound state $u$.  If one of the following conditions is fulfilled
  \begin{itemize}
\item[\textit{(i)}]
    $\alpha_\lm \to \alpha $ for some positive $\alpha$ and $\int_\Real V(\alpha \,\cdot)u^2\,dx<0$;
\item[\textit{(ii)}]
	$\alpha_\lm \to +\infty$,  $u(0)\neq 0$  and $\int_\Real V\,dx<0$;
\item[\textit{(iii)}]
	$\alpha_\lm \to 0$ and $\int_{\Real_-} V\,dx+\theta^2\int_{\Real_+} V\,dx<0$,
  \end{itemize}
then  $H_\lm=H_0+ \lm\alpha_\lm V(\alpha_\lm \cdot)$ possesses the coupling constant threshold $\lm=0$, i.e., for all small positive $\lm$ there exists a negative eigenvalue $e_\lm$ of $H_\lm$ such that $e_\lm\to 0$ as $\lm\to0$.
In addition, $e_\lm=-\lm^2(k^2+o(1))$ as $\lm\to0$,
where
\begin{align*}
  	&k=\frac{\alpha}{\theta^2+1}\int_\Real V(\alpha \,\cdot)u^2\,dx 		\quad  \text{if \ } \alpha_\lm \to \alpha \text{ and } \alpha>0;
  \\
	&k=\frac{u^2(0)}{\theta^2+1}\int_\Real V\,dx \quad \text{if \ } \alpha_\lm \to +\infty;
    \\
    &k=\frac{1}{\theta^2+1}\int_{\Real_-}\kern-5pt V\,dx 	+\frac{\theta^2}{\theta^2+1}
    \int_{\Real_+} \kern-5pt V\,dx \quad \text{if \ } \alpha_\lm \to 0.
\end{align*}
\end{thm}

We will point out two consequences of the theorem.
Note that the trivial potential $U$ has a zero-energy resonance with half-bound state $u=1$. Then $\theta=1$,  $u(0)=1$ and so all conditions \textit{(i)--(iii)} become $\int_\Real V\,dx<0$ as well as all formulas for $k$ convert into $k=\frac12\int_\Real V\,dx$. We have proved:
\begin{cor}
Assume $U=0$. If sequence $\alpha_\lm$ has a finite or infinite limit as $\lm\to 0$ and $\int_\Real V\,dt<0$, then operator $-\frac{d^2}{dx^2}+ \lm\alpha_\lm V(\alpha_\lm \cdot)$ possesses the coupling constant threshold $\lm=0$ with the  eigenvalue
\begin{equation*}
  e_\lm=-\frac{\lm^2}{4}\left(\int_\Real V\,dx\right)^2+o(\lm^2)
  \quad \text{as } \lm\to 0.
\end{equation*}
\end{cor}

It is worth mentioning here that Theorem XIII.110 in \cite[p.338]{ReedSimonIV} states in particular that in the case $V\in C^\infty_0(\Real)$ the operator
$-\frac{d^2}{dx^2}+ \lm V(x)$ has a negative eigenvalue for all positive $\lm$ if   $\int_{\Real}V\,dx<0$. Moreover, in that case this eigenvalue is analytic in $\lm$ at $\lm=0$.
The proof of this assertion is based on  the analyticity of some determinants. Theorem \ref{TheoremOrder2} and the corollary deal with the non-analytic family of operators  and gives us  an example of the non-analytic threshold behaviour.

The theorem also contains the cases when the coupling constant $\lm \alpha_\lm$ is infinitely large. In this instance one could consider that the large coupling constant is neutralized by the rapid  localization of the short range potential.
\begin{cor}
   If $H_0$ has a zero-energy resonance with normalized half-bound
   state~$u$, $u(0)\neq 0$ and $\int_\Real V\,dx<0$, then the operators
   \begin{equation*}
     -\frac{d^2}{dx^2}+U(x)+ \frac{1}{\lm^{\kappa-1}}\,V\left(\frac{x}{\lm^\kappa}\right),\quad \kappa\geq 1,
   \end{equation*}
possess the coupling constant threshold $\lm=0$ with the eigenvalue $e_\lm$ obeying
\begin{equation*}
  e_\lm=-\lm^2\left(\frac{u^2(0)}{\theta^2+1}\int_\Real V\,dx\right)^2+o(\lm^2)
  \quad \text{as } \lm\to 0.
\end{equation*}
In particular,  this is true for the operator $-\frac{d^2}{dx^2}+U(x)+ V\left(\frac{x}{\lm}\right)$.
\end{cor}

Theorem~\ref{TheoremOrder2} describes  the absorption of an eigenvalue by the continuous spectrum with the rate $c\lm^2$.
However the absorption is also possible at  higher rates, when  the inequalities in  conditions \textit{(i)--(iii)} turn into equalities.

\begin{thm}\label{TheoremAlphaPositive}
Assume $V\in W_2^1(\Real)$ and $\alpha_\lm \to\alpha$ as $\lm\to 0$ for
some $\alpha>0$. Suppose also that $H_0$ has a zero-energy resonance and $u$ is a normalized half-bound state. If
\begin{equation}\label{cndxV}
   \int_\Real V(\alpha \,\cdot)\,u^2\,dx=0, \qquad
   \int_\Real xV'(\alpha \,\cdot)\,u^2\,dx<0,
\end{equation}
and $\lm^{-1/3}(\alpha_\lm-\alpha)\to \infty$ as $\lm\to 0$, then $H_\lm=H_0+ \lm\alpha_\lm V\big(\alpha_\lm x\big)$ has the coupling constant threshold $\lm=0$,   and  the corresponding eigenvalue admits asymptotics
\begin{equation*}
e_\lm=-\lm^2(\alpha_\lm-\alpha)^2 \left(\left(\frac{\alpha}{\theta^2+1}\int_\Real xV'(\alpha \,\cdot)\,u^2\,dx\right)^2+o(1)\right)\quad\text{as }\lm\to 0.
\end{equation*}
\end{thm}

The next theorem describes the case of  threshold behaviour, when the support of perturbation $\lm\alpha_\lm V(\alpha_\lm \cdot)$ is small and shrinks to a point in the limit.

\begin{thm}\label{TheoremAlphaGoesToInfty}
Assume  $\alpha_\lm \to +\infty$ and operator $H_0$ possesses a zero-energy resonance with normalized half-bound state $u$. If
   \begin{equation}\label{CnduuxV}
   \int_\Real V\,dx=0, \qquad
   u(0)u'(0)\int_\Real xV\,dx<0,
\end{equation}
and $\alpha_\lm=o(\lm^{-1/3})$ as $\lm\to 0$, then $H_\lm$ has the coupling constant threshold $\lm=0$;   the infinitely small eigenvalue obeys
\begin{equation*}
  e_\lm=-\frac{\lm^2}{\alpha_\lm^2}
  \left(\left(\frac{2 u(0)u'(0)}{\theta^2+1}\int_\Real xV(x)\,dx\right)^2+o(1)\right)\quad\text{as }\lm\to 0.
\end{equation*}
\end{thm}

Let us introduce the function
\begin{equation}\label{Theta}
  \Theta(x)=
  \begin{cases}
    1& \text{for } x<0,\\
    \theta &\text{for } x>0.
  \end{cases}
\end{equation}

\begin{thm}\label{TheoremAlphaGoesToZero}
Assume  that $\alpha_\lm \to 0$, potential $V$ is continuous at the origin, operator $H_0$ has a zero-energy resonance and $u$ is a normalized half-bound state. Suppose also the following conditions hold
   \begin{equation}\label{CndU2Th4}
\int_{\Real_-}\kern-5pt V\,dx+\theta^2\int_{\Real_+}\kern-5pt V\,dx=0, \qquad
   V(0)\int_\Real (u^2-\Theta^2)\,dx<0,
\end{equation}
and $\lm^{1/4}\alpha_\lm^{-1}\to 0$ as $\lm\to 0$.
Then  $H_\lm$ has the coupling constant threshold $\lm=0$ with the eigenvalue $e_\lm$ obeying
\begin{equation*}
  e_\lm=-\lm^2\alpha_\lm^2 \left(\left(\frac{V(0)}{\theta^2+1}\int_\Real (u^2-\Theta^2)\,dx\right)^2+o(1)\right)\quad\text{as }\lm\to 0.
\end{equation*}
\end{thm}

The last theorem concerns the absorption of an eigenvalue by the continuous spectrum when the support of perturbation $\lm\alpha_\lm V(\alpha_\lm \cdot)$ extends to the whole line in the limit.

\begin{rem}
  The assertions of Theorems \ref{TheoremAlphaPositive}-\ref{TheoremAlphaGoesToZero} are not valid for any potential $V$, if  $U$ is trivial. Indeed, in this case $u=1$ and  we see that $u'(0)=0$ and $u^2-\Theta^2=0$ in \eqref{CnduuxV} and \eqref{CndU2Th4} respectively.
  As for \eqref{cndxV}, after integrating by parts we obtain
 \begin{equation*}
   \int_\Real x V'(\alpha x)\,dx=-\frac{1}{\alpha}\int_\Real V(\alpha x)\,dx<0
\end{equation*}
that contradicts to condition $\int_\Real V(\alpha x)\,dx=0$. Consequently, these results can not be compared with those of Simon, when the small eigenvalue is analytic at $\lm=0$.
\end{rem}

\begin{rem}
 The assumption $\lm^{-1/3}(\alpha_\lm-\alpha)\to \infty$ in Theorem~\ref{TheoremAlphaPositive}, as well as the similar assumptions $\alpha_\lm=o(\lm^{-1/3})$ and $\lm^{1/4}\alpha_\lm^{-1}\to 0$  in Theorems~\ref{TheoremAlphaGoesToInfty} and ~\ref{TheoremAlphaGoesToZero} respectively, is of a purely technical nature and is related to  our approximation method for eigenfunctions. For instance, Theorem~\ref{TheoremAlphaPositive} does not cover the case $\alpha_\lm=1$ which has been studied by Klaus \cite{Klaus:1982}  and in which the threshold phenomenon exists. To drop or at least to weaken these conditions, we  have to make up  some additional assumptions on the limiting behaviour of $\alpha_\lm$;  we must also look for more precise asymptotic approximation to the negative  eigenvalue and the corresponding eigenfunction. 
\end{rem}


\section{Proof of Theorem \ref{TheoremOrder2}}

We start with an auxiliary assertion. Let us consider operator $\cH(\kappa,\beta)$ defined by  $\cH(\kappa,\beta)\phi=-\phi''$ on functions $\phi$ in $W_2^2(\Real\setminus\{0\})$, subject to the coupling conditions at the origin
\begin{equation}\label{PointInteraction}
    \begin{pmatrix}
         \phi(+0)\\ \phi'(+0)
    \end{pmatrix} =
    \begin{pmatrix}
            \kappa & 0 \\
            \beta & \kappa^{-1}
    \end{pmatrix}
    \begin{pmatrix}
      \phi(-0)\\ \phi'(-0)
    \end{pmatrix}.
\end{equation}

\begin{lem}\label{LemCHeiv}
If $\kappa \beta<0$, then operator $\cH(\kappa,\beta)$ possesses a unique eigenvalue $$E=-\left(\frac{\kappa\beta}{\kappa^2+1}\right)^2.$$
\end{lem}
\begin{proof}
 We look for a nontrivial $L^2(\Real)$-solution of equation $-\phi''+\omega^2\phi=0$ satisfying \eqref{PointInteraction}. If such a solution exists, then it has the form $\phi(x)=c_1 e^{\omega x}$ for $x<0$ and  $\phi(x)=c_2 e^{-\omega x}$ for $x>0$, where $\omega$ is
 positive. Substituting $\phi$ into \eqref{PointInteraction} yields
\begin{equation*}
\begin{pmatrix}
  \kappa &  -1\\
  \beta+\kappa^{-1}\omega & \omega
\end{pmatrix}
\begin{pmatrix}
  c_1\\ c_2
\end{pmatrix}
=0.
\end{equation*}
Since $\phi$ is a non-zero function, the matrix must be  degenerate. Hence
\begin{equation*}
(\kappa+\kappa^{-1})\,\omega+\beta=0.
\end{equation*}
The last equation admits a positive solution $\omega=-\frac{\kappa\beta}{\kappa^2+1}$ if  $\kappa\beta<0$.
Therefore  $\cH(\kappa,\beta)$ has  negative eigenvalue  $E=-\left(\frac{\kappa\beta}{\kappa^2+1}\right)^2$ with the eigenfunction
$\phi(x)=e^{\omega x}$ if $x<0$ and
$\phi(x)=\kappa e^{-\omega x}$ if $ x>0$.
\end{proof}

To prove Theorem~\ref{TheoremOrder2}, apply the convergence results for
Schr\"{o}dinger operators with $(a \delta'+b \delta)$-like potentials. Let us consider the two-parameter family of operators
  \begin{equation*}
    S_\en= -\frac{d^2}{dt^2}+\lm^{-2}U(\lm^{-1}t)
    +\nu^{-1}V(\nu^{-1}t), \quad
\dmn S_\en=W_2^2(\mathbb{R}).
\end{equation*}
In \cite{GolovatyIEOT2013} the norm resolvent convergence of $S_\en$ was established as positive parameters $\lm$ and $\nu$ tend to zero simultaneously. If operator $H_0$ possesses a zero-energy resonance, then there exist $3$~different cases of the limiting behaviour of $S_\en$ depending on the limit of ratio $\nu/\lm$. Rescaling $x=t/\lm$ of the coordinate yields
  \begin{equation*}
    S_\en= \lm^{-2}\left(-\frac{d^2}{dx^2}+U(x)
    +\lm^2\nu^{-1}V(\lm\nu^{-1}x)\right)
    =\lm^{-2}\left(H_0+\lm^2\nu^{-1}V(\lm\nu^{-1}x)\right).
\end{equation*}
If we will connect the parameters $\nu$ and $\lm$ by relation
$\nu_\lm=\lm/\alpha_\lm $, then $S_\lm=\lm^{-2}H_\lm$,  where $S_\lm$ stands for $S_{\lm, \nu_\lm}$.

First assume  $\alpha_\lm =\frac{\lm}{\nu_\lm}\to \alpha$  as $\lm\to 0$ and $\alpha>0$. In this case $S_\lm$ converge in the norm resolvent sense
to operator $\cH(\theta,\beta)$ associated with point interaction \eqref{PointInteraction}, where $\kappa=\theta$ and
\begin{equation*}
 \beta =\alpha\theta^{-1} \int_\Real V(\alpha x)\, u^2(x)\,dx,
\end{equation*}
(see \cite[Theorem 4.1]{GolovatyIEOT2013}). In view of Lemma~\ref{LemCHeiv}, operator $\cH(\theta,\beta)$ has a negative eigenvalue $E$, because
\begin{equation*}
  \theta\beta=\alpha\, \int_\Real V(\alpha x)\, u^2(x)\,dx<0
\end{equation*}
by  assumption (i) of the theorem. In addition, $E=-k^2$ with
\begin{equation*}
k=\frac{\theta\beta}{\theta^2+1}=\frac{\alpha}{\theta^2+1}\int_\Real V(\alpha x)\, u^2(x)\,dx.
\end{equation*}
The norm resolvent convergence $S_\lm\to \cH(\theta,\beta)$ implies that there exists an eigenvalue $E_\lm$ of $S_\lm$  such that $E_\lm\to E$ as $\lm\to 0$. Since $S_\lm= \lm^{-2}H_\lm$, operator $H_\lm$ possesses a negative eigenvalue $e_\lm=\lm^2E_\lm$ such that
$e_\lm=-\lm^2(k^2+o(1))$ as  $\lm\to 0$.

Now let $\alpha_\lm \to \infty$  as $\lm\to 0$. In view of Theorem~5.1 \cite{GolovatyIEOT2013}, operators $S_\lm$ converge in the norm resolvent sense to $\cH(\theta,\beta)$, where
\begin{equation*}
 \beta =\theta^{-1}u^2(0) \int_\Real V\,dx.
\end{equation*}
The limit operator admits a negative eigenvalue $E=-k^2$, since
\begin{equation*}
  \theta\beta=u^2(0)\,\int_\Real V\,dx<0
\end{equation*}
by assumption. In this case we have
\begin{equation*}
  k=\frac{\theta\beta}{\theta^2+1}=\frac{u^2(0)}{\theta^2+1}\int_\Real V\,dx.
\end{equation*}
Therefore  $H_\lm= \lm^{2}S_\lm$ possesses a negative eigenvalue of order $O(\lm^2)$ as $\lm\to 0$.

Finally, in the case $\alpha_\lm \to 0$ we apply Theorem~3.1 \cite{GolovatyIEOT2013}. The family $S_\lm$ converges to $\cH(\theta,\beta)$ as $\lm\to 0$ in the sense of uniform convergence of resolvents, where
\begin{equation*}
   \beta =\theta^{-1}\,\int_{\Real_-}\kern-5ptV\,dt
    +\theta \,\int_{\Real_+}\kern-5ptV\,dt.
\end{equation*}
Condition \textit{(iii)} of the theorem ensures the existence of  negative eigenvalue $E=-k^2$ for $\cH(\theta,\beta)$ with
\begin{equation*}
  k=\frac{1}{\theta^2+1}\int_{\Real_-} \kern-5pt V\,dx +
\frac{\theta^2}{\theta^2+1}\int_{\Real_+}\kern-5pt V\,dx,
\end{equation*}
because  $\theta\beta=\int_{\Real_-}\kern-4ptV\,dt +\theta^2 \,\int_{\Real_+}\kern-4ptV\,dt<0$.
Hence there exists a small negative eigenvalue $e_\lm$ of $H_\lm$.


\section{Proof of Theorem  \ref{TheoremAlphaPositive}}\label{SecTh2}

We will apply the method of quasimodes.
Let $A$ be a self-adjoint operator in a Hilbert space $L$.
\begin{defn}
  We say a pair $(\mu, \phi)\in \Real\times \dmn A$ is a \textit{quasimode} of  operator $A$ with accuracy $\delta$, if $\|\phi\|_L=1$ and $\|(A-\mu)\phi\|_L\leq\delta$.
\end{defn}

\begin{lem}\cite[p.139]{PDEVinitiSpringer}\label{LemQuasimodes}
   Assume $(\mu, \phi)$ is a quasimode of $A$ with accuracy $\delta>0$  and  the spectrum of $A$ is discrete in  interval
$[\mu-\delta, \mu+\delta]$. Then there exists an eigenvalue $\lambda$ of  $A$ such that $|\lambda-\mu|\leq\delta$.
\end{lem}

The proof of the lemma is simple. Indeed, if  $\mu\not\in \sigma(A)$, then  the distance $d_\mu$ from $\mu$ to the spectrum of $A$  is computed as
\begin{equation*}
  d_\mu^{-1}=\|(A-\mu I)^{-1}\|
  =\sup_{\psi\neq0}\frac{\|(A-\mu I)^{-1}\psi\|_L}{\|\psi\|_L}.
\end{equation*}
Here $\psi$ is an arbitrary vector of $L$. Taking $\psi=(A-\mu)\phi$, we deduce
\begin{equation*}
  d_\mu^{-1}\geq \frac{\|\phi\|_L}{\|(A-\mu I)\phi\|_L}\geq \delta^{-1},
\end{equation*}
and therefore $d_\mu\leq \delta$, from which the assertion
of the lemma follows.

Without loss  of generality we here and henceforth assume that
the supports of potentials $U$ and $V$ lie within $\cI=(-b,b)$ for some $b>0$.
Then a half-bound state $u$ is  constant outside $\cI$ and its restriction to $\cI$ is a non-trivial solution of  the problem
\begin{equation}\label{NeumanProblemWithAlpha}
     - u'' +Uu= 0, \quad t\in \cI,\qquad u'(-b)=0, \quad u'(b)=0.
\end{equation}
Moreover, since $u$ is a normalized half-bound state,  $u(-b)=1$ and  $u(b)=\theta$.

We will construct a family of quasimodes $(-\omega^2_\lm, \phi_\lm)$ of $H_\lm$  as follows. Let as introduce infinitesimal $\eps_\lm= \alpha_\lm -\alpha$ as $\lm\to 0$ and  set  $\omega_\lm=\lm \eps_\lm k_\lm$ for  positive $k_\lm$.  We also write $\phi_\lm=\psi_\lm/\|\psi_\lm\|$, where
 \begin{equation*}
   \psi_\lm(x)=
   \begin{cases}
     e^{\omega_\lm(x+b)}& \text{for } x<-b,\\
     u(x)+\lm v(x)+\lm  \eps_\lm w_{\lm}(x)  & \text{for } |x|<b,\\
     a_{0,\lm}\,e^{-\omega_\lm(x-b)}+a_{1,\lm}\rho(x-b)& \text{for } x>b.
   \end{cases}
 \end{equation*}
 Here and subsequently, $\|\,\cdot\,\|$ stands for the norm in $L_2(\Real)$.
 Suppose that $v$ and $w_\lm$ are  solutions of the problems
 \begin{align}\label{CPv1}
   &-v''+Uv=-V_0u, \qquad v(-b)=0, \;\; v'(-b)=0;
 \\ \label{CPv2}
   &-w''+Uw=-\eps_\lm ^{-1}(V_\lm-V_0)u, \qquad w(-b)=0, \;\; w'(-b)=k_\lm
\end{align}
respectively; $V_0=\alpha V(\alpha \,\cdot)$ denotes the limit of potentials $V_\lm=\alpha_\lm V(\alpha_\lm \cdot)$ as $\lm\to 0$ and $a_{j,\lm}$ are some real quantities.
Note that  the support of $V(\alpha_\lm  \,\cdot)$  lies within $\cI$ for $\lm$ small enough as well as $\eps_\lm$ is different from zero, because  $\lm^{-1/3}\eps_\lm\to \infty$.
Function $\rho\colon \Real_+\to \Real$ is smooth, $\rho(0)=0$, $\rho'(0)=1$ and $\rho(x)=0$ for $x\geq 1$.  This function will be used to correct the discontinuity of $\psi'_\lm$.
We also note that $u$, $v$ and $w_\lm$ belong to $W_2^2(\cI)$ as solutions of $-y''+Uy=f$ with $f\in L_2(\cI)$.

Let us show that $a_{0,\lm}$,  $a_{1,\lm}$ and $k_\lm$ can be chosen  in such a way that $\psi_\lm$ will belong to $\dmn H_\lm$.
By construction, $\psi_\lm$ is continuously differentiable at point $x=-b$. We now  make $\psi_\lm$ and $\psi_\lm'$ continuous at $x=b$.
Since $u(b)=\theta$ and $\rho(0)=0$, we have
$[\psi_\lm]_b=a_{0,\lm}-\theta-\lm v(b)-\lm  \eps_\lm w_\lm(b)$,
where  $[\,\cdot\,]_x$ denotes the jump of a function at the point $x$. Set $a_{0,\lm}=\theta+\lm v(b)+\lm  \eps_\lm w_\lm(b)$. Next, we derive
\begin{multline*}
  [\psi'_\lm]_b=-\omega_\lm a_{0,\lm}+a_{1,\lm}\rho'(0)-\lm v'(b)-\lm  \eps_\lm w_\lm'(b)\\
  = -\lm v'(b)-\lm \eps_\lm \big(w_\lm'(b)+\theta k_\lm\big) +a_{1,\lm}-\lm^2 \eps_\lm k_\lm v(b)-\lm^2  \eps_\lm^2 k_\lm w_\lm(b),
\end{multline*}
since  $\rho'(0)=1$.
Multiplying the equation in  \eqref{CPv1} by half-bound state $u$ and integrating by parts twice over $\cI$ yield
\begin{equation*}
  \theta v'(b)=\int_{\cI}V_0u^2\,dx
  =\alpha\int_{\Real}V(\alpha x)u^2(x)\,dx,
\end{equation*}
and hence $v'(b)=0$, by \eqref{cndxV}.
Applying the same considerations to \eqref{CPv2}, we obtain
\begin{equation}\label{ThetaV'th2}
    \theta w_\lm'(b)-k_\lm=\eps_\lm ^{-1}\int_{\Real} (V_\lm-V_0)u^2\,dx.
\end{equation}
Note that we will often replace integrals over supports of integrands with integrals over $\Real$ and vice versa without commenting on it.
In order to obtain the continuity of $\psi'_\lm$ at $x=b$, we  set
\begin{equation}\label{KThetaV2}
 w_\lm'(b)=-\theta k_\lm, \qquad a_{1,\lm}=\lm^2 \eps_\lm k_\lm \,\big(v(b)+\eps_\lm  w_\lm(b)\big),
\end{equation}
since $v'(b)=0$.
Also, combining \eqref{ThetaV'th2} and the first equality in \eqref{KThetaV2} yields
\begin{equation}\label{KTheta1}
   k_\lm=-\frac{1}{\eps_\lm(\theta^2+1)} \int_{\Real} (V_\lm-V_0)u^2\,dx.
\end{equation}

\begin{prop}\label{PropAsympKlm}
 Under the assumptions of Theorem~\ref{TheoremAlphaPositive}, the value $k_\lm$, given by \eqref{KTheta1}, admits asymptotics
 \begin{equation*}
   k_\lm=-\frac{\alpha}{\theta^2+1}\int_\Real xV'(\alpha \,\cdot)\,u^2\,dx+o(1)\quad\text{as }\lm\to 0.
 \end{equation*}
\end{prop}
\begin{proof}
Since  $\alpha_\lm=\alpha+\eps_\lm$, we have
\begin{multline*}
   \eps_\lm^{-1}(V_\lm(x)-V_0(x))=\eps_\lm^{-1}(\alpha_\lm V(\alpha_\lm x)   -\alpha V(\alpha x))\\
      = \eps_\lm^{-1}((\alpha+\eps_\lm) V(\alpha_\lm x)
   -\alpha V(\alpha x))
     =V(\alpha_\lm x)+\alpha \eps_\lm^{-1}(V(\alpha_\lm x)-V(\alpha x)).
  \end{multline*}
Then we deduce from the continuity of the finite difference operator  that
\begin{multline*}
  V(\alpha_\lm x)+\alpha\eps_\lm^{-1}(V(\alpha_\lm x)-V(\alpha x))\\
  =V(\alpha_\lm x)+
 \alpha x\, \frac{V(\alpha_\lm x)-V(\alpha x)}{(\alpha_\lm -\alpha)x}
  \to V(\alpha x)+\alpha xV'(\alpha x)\quad\text{as }\lm\to 0
\end{multline*}
in $L^2(\Real)$, provided $V\in W_2^1(\Real)$. Recalling $\int_\Real V(\alpha \,\cdot)\,u^2\,dx=0$, we obtain
\begin{equation*}
  \eps_\lm^{-1} \int_{\Real} (V_\lm-V_0)u^2\,dx\to \int_{\Real}(V(\alpha \,\cdot)+\alpha xV'(\alpha \,\cdot))u^2\,dx
  =\alpha \int_{\Real}xV'(\alpha \,\cdot)u^2\,dx,
\end{equation*}
which finishes the proof.
\end{proof}

In view of the proposition above, the values of $k_\lm$, and therefore $\omega_\lm$, are positive for $\lm$ small enough, if the inequality in \eqref{cndxV} holds. Hence $\psi_\lm$ is a $L_2(\Real)$-function, because both the exponents $e^{\pm\omega_\lm(x\pm b)}$ decrease as $x\to \mp\infty$ respectively.
Moreover we chose $a_{0,\lm}$ and $a_{1,\lm}$ so that
$\psi_\lm$  belongs to $W_{2}^2(\Real)$, and hence $\psi\in\dmn H_\lm$.
Since $ \|w_\lm\|_{W_2^2(\cI)}\leq c_1(|k_\lm|+\eps_\lm ^{-1}\|(V_\lm-V_0)\|_{L_2(\cI)})\leq c_2$ by Proposition~\ref{PropAsympKlm}, we  have
\begin{equation}\label{NormPsiBB}
\|\psi_\lm\|_{L^2(\cI)}\leq c_1.
\end{equation}
In addition, an easy computation shows that
\begin{equation}\label{NormPsi}
 c\omega_\lm^{-1/2}\leq\|\psi_\lm\|\leq C \omega_\lm^{-1/2}
\end{equation}
with some constants $c$ and $C$ being independent of $\lm$.

\begin{lem}\label{lemOmegaPhiIsQM}
   The pair $(-\omega_\lm^2 , \phi_\lm)$ is a quasimode of  $H_\lm$ with accuracy $o(\omega_\lm^2)$ as $\lm\to 0$.
\end{lem}
\begin{proof}
If we set $r_\lm=(H_\lm+\omega^2_\lm)\psi_\lm$,  then  $(H_\lm+\omega^2_\lm)\phi_\lm =\|\psi_\lm\|^{-1}r_\lm$.
We must estimate the $L_2$-norm of remainder $r_\lm$ in order to obtain the accuracy of the quasimodes. Let us first suppose $|x|>b$. Since the exponents $e^{\pm\omega_\lm(x\pm b)}$ are exact solutions of $-\psi''+\omega^2_\lm\psi=0$ and $\supp \rho=[0,1]$, we have
\begin{equation}\label{RlmTh2}
 r_\lm(x)=
    -a_{1,\lm}(\rho''(x-b)-\omega^2_\lm\rho(x-b))\quad\text{for }b\leq x\leq b+1
\end{equation}
and $r_\lm(x)=0$ for other $x$ from  set $\{x\colon |x|>b\}$.
Since $\rho$ and $\rho''$ are bounded on $[0,1]$,  $r_\lm$ is of order $a_{1,\lm}$ as $\lm\to 0$. Namely,
\begin{equation}\label{EstRx>b}
  \max_{|x|\geq b}|r_\lm(x)|=\max_{b\leq x\leq b+1}|r_\lm(x)|\leq c_1|a_{1,\lm}|\leq c_2 \lm^2 \eps_\lm,
\end{equation}
by \eqref{KThetaV2}. Next, recalling \eqref{NeumanProblemWithAlpha}--\eqref{CPv2}, we derive
\begin{multline*}
   r_\lm=\left(-\frac{d^2}{dx^2}+U+\lm V_\lm+\omega^2_\lm\right) \psi_\lm
   = (-u''+Uu)+\lm (-v''+Uv+V_0 u)\\
   +\lm  \eps_\lm \big(-w_\lm''+Uw_\lm+\eps_\lm ^{-1}(V_\lm-V_0)u\big)
   +\lm^2V_\lm v+ \lm^2\eps_\lm V_\lm w_\lm
   +\omega^2_\lm \psi_\lm\\
   =\lm^2\left(V_\lm v+ \eps_\lm V_\lm w_\lm
   +\eps_\lm^2 k_\lm^2 \psi_\lm\right)
\end{multline*}
for $|x|<b$.
Hence $\|r_\lm\|_{L_2(\cI)}=O(\lm^2)$ as $\lm\to 0$,
because \eqref{NormPsiBB} holds.
Finally,  $\|r_\lm\|=O(\lm ^2)$ as $\lm\to 0$, by \eqref{EstRx>b}.
Therefore using \eqref{NormPsi}, we obtain
\begin{equation*}
  \|(H_\lm+\omega^2_\lm)\phi_\lm\|
   =\|\psi_\lm\|^{-1}\|r_\lm\|\leq c_1 \lm^{5/2}\eps_\lm^{1/2}
   \leq c_2\omega_\lm^2 (\lm^{1/2}\eps_\lm^{-3/2}).
\end{equation*}
Recall we are assuming $\lm^{-1/3}\eps_\lm\to \infty$.  Hence $\lm^{1/2}\eps_\lm^{-3/2}\to 0$ as $\lm\to 0$, and the lemma follows.
\end{proof}

Since  $(-\omega^2_\lm, \phi_\lm)$ is a quasimode of $H_\lm$ with accuracy $o(\omega^2_\lm)$ and the negative semiaxis  is free of the continuous spectrum, there exists an eigenvalue $e_\lm$ of $H_\lm$ such that $|e_\lm+\omega_\lm^2|=o(\omega_\lm^2)$ as $\lm\to 0$, by Lemma~\ref{LemQuasimodes}.
Therefore the the eigenvalue of $H_\lm$ admits the asymptotics
\begin{equation*}
e_\lm=-\lm^2 \eps_\lm^2 k_\lm^2(1+o(1)) =-\lm^2(\alpha_\lm-\alpha)^2 \left(\left(\frac{\alpha}{\theta^2+1}\int_\Real xV'(\alpha \,\cdot)\,u^2\,dx\right)^2+o(1)\right)
\end{equation*}
as $\lm$ goes to zero, which proves the theorem.


\section{Proof of Theorem  \ref{TheoremAlphaGoesToInfty}}\label{SecTh3}

As in the previous section, the proof consists in the construction of  proper quasimodes  of  $H_\lm$. Since $\alpha_\lm\to \infty$, we set  $\omega_\lm=\lm k_\lm \alpha_\lm^{-1}$ for some positive $k_\lm$ and write $\phi_\lm=\psi_\lm/\|\psi_\lm\|$, where
 \begin{equation*}
   \psi_\lm(x)=
   \begin{cases}
     e^{\omega_\lm(x+b)}& \text{for } x<-b,\\
     u(x)+ \frac{\lm }{\alpha_\lm}\,v_{\lm}(x)+\frac{\lm^2 }{\alpha_\lm^2}\, w_{\lm}(x)  & \text{for } |x|<b,\\
     a_{0,\lm}\,e^{-\omega_\lm(x-b)}+a_{1,\lm}\rho(x-b)& \text{for } x>b.
   \end{cases}
 \end{equation*}
Here $u$ is the normalized half-bound state of $H_0$;
 $v_{\lm}$ and $w_{\lm}$ solve the problems
 \begin{align}\label{CPv1lm}
   &-v''+Uv=-\alpha_\lm V_\lm u, \qquad v(-b)=0, \;\; v'(-b)= k_\lm;
 \\ \label{CPv2lm}
   &-w''+Uw=-\alpha_\lm V_\lm v_{\lm}, \qquad w(-b)=0, \;\; w'(-b)=0
\end{align}
respectively. Recall that $V_\lm=\alpha_\lm V(\alpha_\lm \cdot)$.
As above, we will show that there exist $a_{0,\lm}$, $a_{1,\lm}$ and a positive quantity $k_\lm$ such that $\psi_\lm\in \dmn \cH_\lm$.

The Cauchy data in \eqref{CPv1lm} and \eqref{CPv2lm}, together with conditions $u(-b)=1$ and $u'(-b)=0$, ensure the continuity of $\psi_\lm$ and $\psi'_\lm$ at point $x=-b$. Next, if we set
$a_{0,\lm}=\theta+\tfrac{\lm }{\alpha_\lm}\,v_{\lm}(b)+\tfrac{\lm^2 }{\alpha_\lm^2}\, w_{\lm}(b)$, then $[\psi_\lm]_b=0$.
We also have
\begin{align*}
[\psi'_\lm]_b&=-\omega_\lm a_{0,\lm}+a_{1,\lm}-\tfrac{\lm }{\alpha_\lm}\,v'_{\lm}(b)-\tfrac{\lm^2 }{\alpha_\lm^2}\, w'_{\lm}(b)\\
  &= -\tfrac{\lm }{\alpha_\lm}\big(\theta k_\lm+v'_{\lm}(b)\big)
  +a_{1,\lm}
  -\tfrac{\lm^2k_\lm }{\alpha_\lm^2}\,v_{\lm}(b)-\tfrac{\lm^3k_\lm }{\alpha_\lm^3}\, w_{\lm}(b)-\tfrac{\lm^2 }{\alpha_\lm^2}\, w'_{\lm}(b).
\end{align*}
To achieve $[\psi'_\lm]_b=0$,  assume
\begin{equation}\label{VprimeA1th3}
v'_{\lm}(b)=-\theta k_\lm, \qquad
  a_{1,\lm}=
  \tfrac{\lm^2k_\lm }{\alpha_\lm^2}\big( v_{\lm}(b)+\tfrac{1}{k_\lm}\, w'_{\lm}(b)+\tfrac{\lm  }{\alpha_\lm}\, w_{\lm}(b)\big).
\end{equation}
Multiplying the equation in  \eqref{CPv1lm} by  $u$ and integrating by parts twice yield
\begin{equation}\label{V2primeBlm}
      \theta v_{\lm}'(b)=k_\lm+\alpha_\lm\int_{\Real}V_\lm u^2\,dx.
\end{equation}
Upon  substituting $v'_{\lm}(b)=-\theta k_\lm$ into
the last equality, we  derive
\begin{equation}\label{KlmFormula}
   k_\lm=-\frac{\alpha_\lm}{\theta^2+1}\int_{\Real}V_\lm u^2\,dx.
\end{equation}

The next statement will allow us to construct the asymptotics of $k_\lm$.
\begin{prop}\label{PropIntVlmAsympt}
  Assume that $\int_\Real V\,dx=0$ and $g\in W_{2,loc}^2(\Real)$. Then
  \begin{equation*}
    \alpha_\lm\int_{\Real} V_\lm g\,dx=g'(0)\int_{\Real} xV(x)\,dx+o(\alpha_\lm^{-1/2})
  \end{equation*}
as $\alpha_\lambda\to\infty$.
\end{prop}
\begin{proof}
 Recall we are assuming $\supp V\subset \cI=(-b, b)$ and hence $\supp V(\alpha_\lm \,\cdot)$ lies within the small interval  $\cI_\lm=(-\frac{b}{\alpha_\lm}, \frac{b}{\alpha_\lm})$. Then we deduce
  \begin{multline*}
  \alpha_\lm\int_{\Real}V_\lm g\,dx=
  \alpha_\lm^2\int_{\cI_\lm}V(\alpha_\lm x)g(x)\,dx\\=
  \alpha_\lm\int_{\cI}V(t) g(\tfrac{t}{\alpha_\lm})\,dt
  =
  \alpha_\lm\int_{\cI}V(t) \big(g(0)+g'(0)\tfrac{t}{\alpha_\lm}
  +\beta(\tfrac{t}{\alpha_\lm})\big)\,dt
  \\
  =g'(0)\int_{\cI}tV(t)\,dt+
  \alpha_\lm\int_{\cI}V(t)\beta(\tfrac{t}{\alpha_\lm})\,dt,
\end{multline*}
where $\beta(x)=\int_0^x(x-s)g''(s)\,ds$. Since $g''\in L_2(\cI)$, we have
\begin{equation*}
 |\beta(x)|= \left|\int_0^x(x-s)g''(s)\,ds\right|\leq |x|\left|\int_0^x|g''(s)|\,ds\right|\leq
  \|g''\|_{L_2(\cI)}|x|^{3/2}.
\end{equation*}
Hence $|\beta(\tfrac{t}{\alpha_\lm})|\leq c |t|^{3/2}\alpha_\lm^{-3/2}$ for $t\in\cI$, which completes the proof.
\end{proof}

Let us apply the proposition  to  the integral in the right-hand side of \eqref{KlmFormula}:
\begin{equation}\label{KlmAsympt}
   k_\lm=-\frac{2 u(0)u'(0)}{\theta^2+1}\int_\Real xV(x)\,dx+o(\alpha_\lm^{-1/2})\qquad \text{as } \lm\to 0.
\end{equation}
Hence $k_\lm$ has a finite limit as $\lm\to 0$ and the inequality in \eqref{CnduuxV} ensures that $k_\lm$ and $\omega_\lm$ are positive for $\lambda$ small enough.

As in the proof of  Lemma~\ref{lemOmegaPhiIsQM}, the remainder $r_\lm=(H_\lm+\omega^2_\lm)\psi_\lm$ is given by \eqref{RlmTh2}, provided  $|x|>b$,  and therefore it has the same order of smallness as $a_{1,\lm}$ in \eqref{VprimeA1th3}. We next have for $x\in \cI$
\begin{equation}\label{EstRth3}
\begin{aligned}
   r_\lm&=\left(-\tfrac{d^2}{dx^2}+U+\lm V_\lm+\omega^2_\lm\right) \left(u+ \tfrac{\lm }{\alpha_\lm}\,v_{\lm}+\tfrac{\lm^2 }{\alpha_\lm^2}\, w_{\lm}\right)
   \\
   &= (-u''+Uu)+\tfrac{\lm}{\alpha_\lm} (-v_{\lm}''+Uv_{\lm}+\alpha_\lm V_\lm u)+\tfrac{\lm^2 }{\alpha_\lm^2 } (-w_{\lm}''+Uw_{\lm}
   \\
  & +\alpha_\lm V_\lm v_{\lm})
   +\tfrac{\lm^3}{\alpha_\lm}V_\lm w_{\lm}
   +\omega^2_\lm \psi_\lm
   =\lm^3V(\alpha_\lm\,\cdot) w_{\lm}
   +\omega^2_\lm \psi_\lm.
\end{aligned}
\end{equation}

\begin{prop}\label{PropV1V2Th3}
If $v_{\lm}$ and $w_{\lm}$ are solutions of  \eqref{CPv1lm} and \eqref{CPv2lm} respectively, then
\begin{equation*}
\|v_{\lm}\|_{C^1(\cI)}\leq c_1\alpha_\lm,
\qquad
\|w_{\lm}\|_{C^1(\cI)}\leq c_2\alpha_\lm^2.
\end{equation*}
The constants $c_j$ do not depend on $\lm$.
\end{prop}
\begin{proof}
Let $u_1$ be the solution of $-y''+Uy=0$ such that $u_1(x)=x+b$ for $x\leq -b$. This solution is linearly independent with  $u$. It is easy to check that
 \begin{equation*}
        v_{\lm}(x)=k_\lm u_1(x)+\alpha_\lm\int_{-\frac{b}{\alpha_\lm}}^x K(x,s) V_\lm (s) u(s)\,ds,
 \end{equation*}
where $K(x,s)=u(s)u_1(x)-u(x)u_1(s)$.
Then
\begin{equation*}
\left|\int_{-\frac{b}{\alpha_\lm}}^x \kern-3pt K(x,s) V_\lm (s) u(s)\,ds\right|
\leq\alpha_\lm\int_{\cI_\lm} |K(x,s)|\, |V(\alpha_\lm s)|\, |u(s)|\,ds
\leq c_1\alpha_\lm\int_{\cI_\lm}\,ds \leq c_2,
\end{equation*}
for all $x\in\cI$ and therefore we estimate
\begin{equation*}
    \max_{x\in\cI}|v_{\lm}(x)|\leq |k_\lm| \max_{x\in\cI}|u_1(x)|+\alpha_\lm\left|\int_{-\frac{b}{\alpha_\lm}}^x K(x,s)
     V_\lm(s) u(s)\,ds\right|\leq c_1 \alpha_\lm.
\end{equation*}
We similarly obtain the bound $\max_{x\in\cI}|v'_{\lm}(x)|\leq c_1 \alpha_\lm$, since
 \begin{equation*}
 v'_{\lm}(x)=k_\lm u_1'(x)+\alpha_\lm\int_{-\frac{b}{\alpha_\lm}}^x K'_x(x,s) V_\lm (s) u(s)\,ds.
 \end{equation*}
We also apply the estimates above to  the solution
\begin{equation*}
        w_{\lm}(x)=\alpha_\lm\int_{-\frac{b}{\alpha_\lm}}^x K(x,s) V_\lm(s) v_{\lm}(s)\,ds
 \end{equation*}
of \eqref{CPv2lm} and derive $|w_{\lm}(x)|\leq c \alpha_\lm^2$ for all $x\in \cI$. Finally, we  deduce
\begin{equation*}
    |w'_{\lm}(x)|\leq\alpha_\lm\int_{\cI_\lm} |K'_{x}(x,s)|\, |V_\lm(s)|\, |v_{\lm}(s)|\,ds\leq c \alpha_\lm^2.
\end{equation*}
\end{proof}

Returning now to \eqref{VprimeA1th3} and \eqref{EstRth3}, we have the bounds
\begin{gather*}
  |a_{1,\lm}|\leq
  c_1\lm^2\alpha_\lm^{-2}\big( |v_{\lm}(b)|
  +|w'_{\lm}(b)|+\lm\alpha_\lm^{-1} |w_{\lm}(b)|\big)\leq c_2\lm^2,\\
  \|V(\alpha_\lm \cdot) w_{\lm}\|^2_{L_2(\cI)}=\int_{\cI_\lm}|V(\alpha_\lm \,\cdot)w_{\lm}|^2 \,dx\leq c_3\alpha_\lm^{4}\int_{\cI_\lm} \,dx\leq c_4\alpha_\lm^{3},
\end{gather*}
in view of Proposition~\ref{PropV1V2Th3}.
Therefore $\|r_\lm\|\leq c_4 (\lm^2+\lm^3\alpha_\lm^{3/2})\leq c_5\lm^2$,
by \eqref{NormPsiBB} and the assumption $\alpha_\lm=o(\lm^{-1/3})$ as $\lm\to 0$. Recalling the definition of $\phi_\lm$ and estimate \eqref{NormPsi}, we conclude
\begin{equation*}
  \|(H_\lm+\omega^2_\lm)\phi_\lm\|
   =\|\psi_\lm\|^{-1}\|r_\lm\|\leq c_6\omega_\lm^{1/2}\lm^2
   \leq c_7\omega_\lm^{2}\lm^{1/2}\alpha_\lm^{3/2},
\end{equation*}
where  $\lm^{1/2}\alpha_\lm^{3/2}$ tends to zero.
Therefore the pair $(-\lm^2 k_\lm^2 \alpha_\lm^{-2}, \phi_\lm)$ is a quasimode of $H_\lm$ with accuracy
$o(\omega^2_\lm)$ and  so $H_\lm$ possesses a small negative eigenvalue $e_\lm$ with asymptotics $-\lm^2\alpha_\lm^{-2} (k_\lm^2+o(1)) $ as $\lm\to 0$, where $k_\lm$ is given by \eqref{KlmAsympt}.


\section{Proof of Theorem  \ref{TheoremAlphaGoesToZero}}\label{Proof4}

In the case when  $\alpha_\lm$ goes to zero, the support of $V_\lm$ lies in $\cI_\lm=(-b\alpha_\lm^{-1},b\alpha_\lm^{-1})$ and extends to the whole line as $\lm\to 0$. Therefore we will look for the approximation to an eigenfunction of $H_\lm$ in the form
 \begin{equation*}
   \psi_\lm(x)=
   \begin{cases}
     e^{\omega_\lm(x+b\alpha_\lm^{-1})}& \text{for } x<-b\alpha_\lm^{-1},\\
     u(x)+ \lm \alpha_\lm\,v_{\lm}(x)  & \text{for } |x|<b\alpha_\lm^{-1},\\
     a_{0,\lm}\,e^{-\omega_\lm(x-b\alpha_\lm^{-1})}
     +a_{1,\lm}\rho(x-b\alpha_\lm^{-1})& \text{for } x>b\alpha_\lm^{-1},
   \end{cases}
 \end{equation*}
where $\omega_\lm=\lm\alpha_\lm k_\lm$.
Set $\phi_\lm=\psi_\lm/\|\psi_\lm\|$.
Here, as usual, $u$ is the normalized half-bound state of $H_0$;
 $v_{\lm}$ solves the problems
 \begin{equation}\label{CPv1lmlast}
   -v''+Uv=-V(\alpha_\lm x) u(x), \qquad v(-b\alpha_\lm^{-1})=0, \;\; v'(-b\alpha_\lm^{-1})= k_\lm.
\end{equation}
We assume  $a_{0,\lm}=\theta+\lm \alpha_\lm v_\lm(b\alpha_\lm^{-1})$,
$a_{1,\lm}=\lm^2 \alpha_\lm^2 v_\lm(b\alpha_\lm^{-1})$ and $k_\lm=-\theta^{-1} v'_{\lm}(b\alpha_\lm^{-1})$, which ensures that $\psi_\lm$ and $\psi_\lm'$ have no jump discontinuities at $x=\pm b\alpha_\lm^{-1}$ and therefore $\psi_\lm\in W^2_{2,loc}(\Real)$.
To prove that $\psi_\lm\in\dmn H_\lm$, we must check that $k_\lm$ is positive.
From \eqref{CPv1lmlast} we as above derive
\begin{equation*}
  \theta v_\lm'(b\alpha_\lm^{-1})-k_\lm
  =\int_{\Real}V(\alpha_\lm x )u^2(x)\,dx.
\end{equation*}
Combine this equality and $v'_{\lm}(b\alpha_\lm^{-1})=-\theta k_\lm$, to deduce
\begin{equation}\label{KlmExactLast}
  k_\lm=-\frac{1}{\theta^2+1}\int_{\Real}V(\alpha_\lm x )u^2(x)\,dx.
\end{equation}

\begin{prop}\label{PropV0IntLast}
  Under the assumptions of Theorem~\ref{TheoremAlphaGoesToZero},
  \begin{equation*}
    \int_{\Real}V(\alpha_\lm \,\cdot )u^2\,dx=V(0)\int_\Real (u^2-\Theta^2)\,dx+o(\alpha_\lm)
  \quad\text{as }\alpha_\lm\to 0,
  \end{equation*}
  where $\Theta$ is given by \eqref{Theta}.
\end{prop}

\begin{proof}
Recalling the fact that $u(x)=1$ for $x<-b$ and $u(x)=\theta$ for $x>b$, we have
\begin{multline*}
   \int_{\Real}V(\alpha_\lm x )u^2(x)\,dx=\frac{1}{\alpha_\lm}
   \int_{\Real}V(t)u^2\left(\tfrac{t}{\alpha_\lm}\right)\,dt
   \\
   =
   \frac{1}{\alpha_\lm}\left(\int\limits^{b\alpha_\lm} _{-b\alpha_\lm}V(t)u^2\left(\tfrac{t}{\alpha_\lm}\right)\,dt
   +\int\limits_{-\infty}^{-b\alpha_\lm} V(t) \,dt+\theta^2\int\limits^{+\infty}_{b\alpha_\lm}  V(t)\,dt\right).
\end{multline*}
Now condition $\int_{\Real_-}\kern-3pt V\,dx+\theta^2\int_{\Real_+}\kern-3pt V\,dx=0$ implies
\begin{equation*}
  \int\limits_{-\infty}^{-b\alpha_\lm} V(t) \,dt+\theta^2\int\limits^{+\infty}_{b\alpha_\lm}  V(t)\,dt
  =
  -\int\limits^{0}_{-b\alpha_\lm} V(t) \,dt
   -\theta^2\int\limits_{0}^{b\alpha_\lm}  V(t)\,dt,
\end{equation*}
and we thereby obtain
\begin{multline*}
   \int_{\Real}V(\alpha_\lm x )u^2(x)\,dx
   =
   \frac{1}{\alpha_\lm}\left(
   \int\limits^{b\alpha_\lm} _{-b\alpha_\lm}V(t)u^2\left(\tfrac{t}{\alpha_\lm}\right)\,dt
   -\int\limits^{0}_{-b\alpha_\lm} V(t) \,dt
   -\theta^2\int\limits_{0}^{b\alpha_\lm}  V(t)\,dt
   \right)
   \\
   =\frac{1}{\alpha_\lm}\int\limits_{-b\alpha_\lm}^{b\alpha_\lm}
  V(t)\left(u^2\left(\tfrac{t}{\alpha_\lm}\right)-\Theta^2(t)\right)\,dt
  =\frac{1}{\alpha_\lm}\int\limits_{\Real}
  V(t)\left(u^2\left(\tfrac{t}{\alpha_\lm}\right)-\Theta^2(t)\right)\,dt,
\end{multline*}
since $u(\alpha_\lm^{-1}\cdot)-\Theta$ has a compact support lying within $[-b\alpha_\lm,b\alpha_\lm]$.  Thus
\begin{multline*}
  \int_{\Real}V(\alpha_\lm x )u^2(x)\,dx
   =
  \frac{1}{\alpha_\lm}\int_{\Real}
  V(t)\left(u^2\left(\tfrac{t}{\alpha_\lm}\right)
  -\Theta^2\left(\tfrac{t}{\alpha_\lm}\right)\right)\,dt\\
  =V(0)\int_\Real (u^2-\Theta^2)\,dx+o(\alpha_\lm)
  \quad\text{as }\alpha_\lm\to 0,
\end{multline*}
because $\Theta=\Theta(\alpha_\lm^{-1}\cdot)$,  $V$ is continuous at $x=0$ and $\alpha_\lm^{-1}\left(u^2(\alpha_\lm^{-1}\,\cdot)  -\Theta^2(\alpha_\lm^{-1}\,\cdot)\right)$
is a $\delta$-like sequence.
Recall that $\eps^{-1}\int_\Real w\left(\frac{x}{\eps}\right)
\eta(x) \,dx\to \eta(0)\int_\Real w\,dx$, as $\eps$ goes to zero, for
any $w\in L_1(\Real)$ and  $\eta\in C(\Real)$.
\end{proof}

In view of  Proposition~\ref{PropV0IntLast}, we then obtain from \eqref{KlmExactLast} the asymptotic formula:
\begin{equation*}
  k_\lm
  =-\frac{V(0)}{\theta^2+1}\int_\Real (u^2-\Theta^2)\,dx+o(\alpha_\lm)\quad\text{as }  \lm\to 0.
\end{equation*}
Consequently, $k_\lm$ remains bounded as $\lm\to 0$  and is positive
for small $\lambda$,  by  \eqref{CndU2Th4}.

\begin{prop}\label{PropVlmLast}
There exists a constant $C$ such that
  \begin{equation*}
    |v_\lm(x)|\leq C \alpha_\lm^{-2}
  \end{equation*}
for all $x\in (-b/\alpha_\lm, b/\alpha_\lm)$.
\end{prop}
\begin{proof}
The solution of \eqref{CPv1lmlast} can be written as
\begin{equation*}
        v_\lm(x)= \frac{bk_\lm}{\alpha_\lm}\,u(x)+k_\lm u_1(x)+\int_{-b/\alpha_\lm}^x K(x,s)V(\alpha_\lm s) u(s)\,ds,
\end{equation*}
where $u_1$  and $K$ are the same as in the proof of Proposition~\ref{PropV1V2Th3}. Since
\begin{equation*}
  |u_1(x)|\leq C_1 (|x|+1), \quad |K(x,s)|\leq C_2 (|x|+|s|+1)
\end{equation*}
for all $x, s \in\Real$ and $k_\lm$ is bounded on $\lm$, we derive
\begin{equation*}
|v_\lm(x)|\leq c_1 \alpha_\lm^{-1} +c_2(|x|+1) +c_3\int\limits_{-b/\alpha_\lm}^x (|x|+|s|+1)\,ds\leq C\alpha_\lm^{-2},
\end{equation*}
provided $x$ belongs to $(-b/\alpha_\lm, b/\alpha_\lm)$.
 \end{proof}

As before, in Sections \ref{SecTh2} and \ref{SecTh3}, we must estimate the remainder $r_\lm=(H_\lm+\omega^2_\lm)\psi_\lm$ having the form
\begin{equation*}
  r_\lm(x)=
  \begin{cases}
    \phantom{-}\lm^2\alpha_\lm^2 V(\alpha_\lm x) v_\lm(x)
   +\omega^2_\lm \psi_\lm(x) &\text{for } |x|<\frac{b}{\alpha_\lm}\\
    -\lm^2 \alpha_\lm^2 v_\lm(\frac{b}{\alpha_\lm})
    \big(\rho''(x-\frac{b}{\alpha_\lm})
    -\omega^2_\lm\rho(x-\frac{b}{\alpha_\lm})\big)&\text{for }\frac{b}{\alpha_\lm}\leq x\leq \frac{b}{\alpha_\lm}+1,\\
     \phantom{-}0, &\text{otherwise}.
  \end{cases}
\end{equation*}
Applying Proposition~\ref{PropVlmLast}, we obtain that
\begin{equation*}
  \|\psi_\lm\|\geq c_1\omega_\lm^{-1/2}, \quad \|\psi_\lm\|_{L_2(\cI_\lm)}\leq c_2 \alpha_\lm^{-1/2}, \quad
\|V(\alpha_\lm \cdot) v_\lm\|_{L_2(\cI_\lm)}\leq c_3\alpha^{-5/2},
\end{equation*}
and hence recalling condition $\lm^{1/4}\alpha_\lm^{-1}\to 0$ gives us
\begin{multline*}
     \|r_\lm\|\leq \lm^2\alpha_\lm^2 \|V(\alpha_\lm \cdot) v_\lm\|_{L_2(\cI_\lm)}
   +\omega^2_\lm \|\psi_\lm\|_{L_2(\cI_\lm)}
   \\
   +\lm^2 \alpha_\lm^2 |v_\lm(b\alpha_\lm^{-1})|\|\rho''-\omega^2_\lm\rho\|_{L_2(0,1)}
          \leq c_4\lm^2\alpha_\lm^{-1/2}.
 \end{multline*}
Finally we have
  \begin{equation*}
  \|(H_\lm+\omega^2_\lm)\phi_\lm\|
   =\|\psi_\lm\|^{-1}\|r_\lm\|\leq c_5\omega_\lm^{1/2}\lm^2\alpha_\lm^{-1/2}
   \leq c_6\omega_\lm^{2}\lm^{1/2}\alpha_\lm^{-2},
\end{equation*}
where $\lm^{1/2}\alpha_\lm^{-2}\to 0$ as $\lm\to 0$. Therefore $(-\lm^2\alpha_\lm^2 k_\lm^2, \phi_\lm)$ is a quasimode of $H_\lm$ with accuracy $o(\omega^2_\lm)$. Existence of the quasimode ensures the existence of the eigenvalue $e_\lm=-\omega^2_\lm$, where
\begin{equation*}
  \omega_\lm=\lm\alpha_\lm \left(\frac{V(0)}{\theta^2+1}\int_\Real (u^2-\Theta^2)\,dx+o(1)\right)\quad\text{as }\lm\to 0.
\end{equation*}

\begin{rem}
The proof of Theorem~\ref{TheoremOrder2} is based on
the norm resolvent convergence of Hamiltonians
$-\frac{d^2}{dx^2}+\lm^{-2}U(\lm^{-1}x)+\nu^{-1}V(\nu^{-1}x)$. This convergence  can be also proved for potentials obeying
\begin{equation*}
  \int_\Real(1+|x|)|U(x)|\,dx<\infty, \qquad\int_\Real(1+|x|)|V(x)|\,dx<\infty
\end{equation*}
(see \cite{GolovatyHrynivProcEdinburgh2013}, where the case of $\delta'$-like potential has been treated) and therefore we can extend the class of admissible potentials. This is also true in relation to
the rest results of this paper, because we can use the WKB-approximations of solutions as $|x|\to \infty$ in place of the exponents $a_\lm e^{\pm\omega_\lm x}$ in the structure of quasimodes.
\end{rem}


\end{document}